\documentclass[12pt]{amsart}

\usepackage{tikz-cd}
\usepackage{amssymb}
\usepackage{amsthm}
\usepackage{amscd}
\usepackage[mathscr]{eucal}
\usepackage{enumitem}
\usepackage{centernot}

\usepackage{hyperref}

\newtheorem{Theorem}{Theorem}[section]
\newtheorem{theorem}[Theorem]{Theorem}

\newtheorem{proposition}[Theorem]{Proposition}

\newtheorem{corollary}[Theorem]{Corollary}

\newtheorem{lemma}[Theorem]{Lemma}

\newtheorem{fact}[Theorem]{Fact}
\newtheorem{remark/def}[Theorem]{Remark/Definition}

\newtheorem{claim}[Theorem]{Claim}

\theoremstyle{definition}

\newtheorem{example}[Theorem]{Example}

\newtheorem{remark}[Theorem]{Remark}

\newtheorem{definition}[Theorem]{Definition}

\newtheorem{notation}[Theorem]{Notation}

\newtheorem{question}[Theorem]{Question}
\newtheorem{def/rem}[Theorem]{Definition/Remark}
\newtheorem{not/rem}[Theorem]{Notation/Remark}
\newsavebox{\indbin}
\savebox{\indbin}{\begin{picture}(0,0)
\newlength{\gnu}
\settowidth{\gnu}{$\smile$} \setlength{\unitlength}{.5\gnu}
\put(-1,-.65){$\smile$} \put(-.25,.1){$|$}
\end{picture}}
\def \indo {\mathop{\smile \hskip -0.9em ^| \ }}
\def \depo {\mathop{ \not \smile \hskip -0.9em  ^| \ }}

\newcommand{\be}{\begin{enumerate}}
\newcommand{\bd}{\begin{defn}}

\newcommand{\bt}{\begin{theorem}}
\newcommand{\bl}{\begin{lemma}}
\newcommand{\ee}{\end{enumerate}}
\newcommand{\ed}{\end{defn}}
\newcommand{\et}{\end{theorem}}
\newcommand{\el}{\end{lemma}}

\newcommand{\ov}{\overline}

\newcommand{\BN}{{\mathbb N}}

\newcommand{\id}{\operatorname{id}}
\newcommand{\Aut}{\operatorname{Aut}}

\def\stab{\operatorname{Stab}}

\def\dcl{\operatorname{dcl}}

\def\acl{\operatorname{acl}}

\def\o{\operatorname{o}}
\def\Acl{\operatorname{Acl}}
\def\mcl{\operatorname{\mu Acl}}
\def\mR{\operatorname{\mu R}}
\def\NM{\mathcal{NM}}

\newcommand{\rist}{\mathrm{rist}}

\newcommand{\grp}[1]{\langle #1\rangle}

\title{Geometric stability theory for $\mu$-structures}

\author[Junguk Lee]{Junguk Lee\\ \textnormal{With an appendix by Michael Cohen, and Phillip Wesolek}}

\address{Instytut Matematyczny, Uniwersytet Wrocławski\\ pl. Grunwaldzki 2/4\\        50-384 Wrocław, Poland}
\email{jlee@math.uni.wroc.pl}

\begin{document}

\begin{abstract}
We introduce a notion of $\mu$-structures which are certain locally compact group actions and prove some counterparts of results on Polish structures(introduced by Krupinski in \cite{Kru5}). Using the Haar measure of locally compact groups, we introduce an independence, called $\mu$-independence, in $\mu$-structures having good properties. With this independence notion, we develop geometric stability theory for $\mu$-structures. Then we see some structural theorems for compact groups which are $\mu$-structure. We also give examples of profinite structures where $\mu$-independence is different from $nm$-independence introduced by Krupinski for Polish structures.
\end{abstract}
\maketitle

\section{Introduction}\label{Introduction}
In \cite{N2,N1}, Newelski introduced topological objects called {\em profinite structures}. He introduced $m$-independence on profinite structures and using this independence, developed geometric stability theory on profinite structures. He also compared $m$-independence with another independence called $\mu$-independence. The notion of $m$-independence is defined topologically and the notion of $\mu$-independence is defined measure theoretically. In \cite{Kru5}, Krupinski generalized this topological independence notion to more wide contexts, called {\em Polish structures}. A {\em Polish structure} is a pair $(X,G)$ where $G$ is a Polish group acting faithfully on a set $X$ such that the stabilizer of each $x\in X$ is a closed subgroup of $G$. For a finite tuple $a$ of $X$ and $A\subset X$, we denote the orbit of $a$ under the action of the pointwise stabilizer  $G_A$ of $A$ by $\o(a/A)$, that is, $\o(a/A):=\{ga|\ g\in G_A\}$. There is a canonical surjective map $\pi_{A,a}:\ G_A\rightarrow \o(a/A)$ defined by $g\mapsto ga$. If there is no confusion, we write $\pi_A$ for $\pi_{A,a}$. For a Polish structure $(X,G)$, he defined a well-behaved independence, called {\em $nm$-independence}. For $A,B\subset X$ finite, we say a finite tuple $a$ of $X$ is {\em $nm$-independent} from $B$ over $A$ if $\pi_A^{-1}[\o(a/AB)]$ is non-meager in $\pi_A^{-1}[\o(a/A)]$, denoted by $a\indo^{nm}_A B$. This independence satisfies the following properties:
\begin{itemize}
	\item (Invariance) For all $g\in G$, $a\indo^{nm}_A B\Leftrightarrow ga\indo^{nm}_{g[A]}g[B]$.
	\item (Symmetry) $a\indo^{nm}_C b \Leftrightarrow b\indo^{nm}_C a$.
	\item (Transitivity) Assume $A\subset B\subset C$. Then $a\indo^{nm}_A B$ and $a\indo^{nm}_B C$ if and only if $a\indo^{nm}_A C$.
\end{itemize}
We call a $G$-set $X$ is {\em small} if for every finite set $A\subseteq X$, the action of the pointwise stabilizer $G_A$ on $X$ has countably many orbits.
\begin{itemize}
	\item (Extension) Suppose $(X,G)$ is small. Then there is $a'\in \o(a/A)$ such that $a'\indo^{nm}_A B$.
\end{itemize}
\noindent  We call the above $4$ properties {\bf the basic 4 axioms}. He developed model theory for Polish structures. He introduced a notion of imaginaries. For a given Polish structure $(X,G)$, after extending $X$ to $X^{eq}$ by adding imaginaries, the $G$-set $(X^{eq},G)$ is still a Polish structure and $nm$-independence is well-extended to $X^{eq}$. Using $nm$-independence, he defined the $nm$-rank on orbits of finite tuples of $X^{eq}$ over finite parameter sets from $X^{eq}$. A Polish structure having ordinal $nm$-ranks is called {\em $nm$-stable}. Using model theoretic method with $nm$-independence and the $nm$-rank, he gave a structure theorem for $nm$-stable Polish structures $(X,G)$ where $G$ acts on a compact group $X$ as homeomorphisms.

Here we introduce {\em $\mu$-structures} by considering locally compact group actions instead of Polish group actions and we define an independence relation on $\mu$-structures induced from the Haar measure on locally compact groups, which generalizes the notion of $\mu$-independence in profinite structures. We also call such independence for $\mu$-structures $\mu$-independence.
\begin{definition}\label{mu_structure}
A {\em $\mu$-structure} is a faithful $G$-action on a set $X$ where $G$ is {\bf  a compact group or a locally compact Polish group}, and for each $x\in X$, the stabilizer of $x$, $G_x:=\{g\in G|\ gx=x\}$ is closed in $G$.
\end{definition}
\noindent Note that in Definition \ref{mu_structure}, if $G$ is compact, it needs not be a Polish group. Roughly speaking, $a$ is `independent' from $B$ over $A$ if $\pi_A^{-1}[\o(a/AB)]$ is a `large' subset of $\pi_A^{-1}[\o(a/A)]$. In Polish structures, being large is measured by being non-meager. For $\mu$-structures, locally compact groups are equipped with the Haar measures and we can measure a size of a subset with the Haar measures. Using the Haar measure, we will introduce a ternary relation, called {\em $\mu$-independence}, and develop counterparts of some results on Polish structures. We will show that $\mu$-independence also satisfies the basic $4$ axioms. With $\mu$-independence, we will develop geometric stability theory for the $\mu$-structure, and see some structure theorems for $\mu$-stable structure $(X,G)$ where $G$ acts on a compact group $X$  as homeomorphisms. Also we compare two notions of $\mu$-independence and $nm$-independence in the case of compact structures. For its generality of Polish structure, there are many topological spaces which can be concerned as Polish structure(See \cite[Section 4]{Kru5}). Any profinite structure is interpretable in a first order theory(See \cite[Theorem 1.4]{Kru4}), and any profinite group is concerned as a profinite structure(See \cite[Example 1]{N2}). Note that $\mu$-structures generalize profinite structures because profinite groups are compact groups. It is interesting to find profinite groups acting on profinite spaces, which are small profinite structures. In Appendix, Cohen and Wesolek give a wide class of profinite groups, called {\em profinite branch groups}, which act on rooted trees. They showed that this group action induces a small profinite structure on the boundary of a rooted tree.

In Section \ref{Introduction}, we recall some basic facts on the Haar measure on locally compact groups and some notions for basic model theory for $\mu$-structures. In Section \ref{section mu-independence}, we will show that $\mu$-independence satisfies the basic $4$ axioms. Invariance, symmetry, and extension comes easily from the basic properties of the Haar measure, for example, the quasi-invariance and the $\sigma$-additivity. We will mainly focus on proving transitivity of $\mu$-independence. In Section \ref{Basic model theory for mu-structure}, we introduce the $\mu$-rank and $\mu$-stability for $\mu$-structures. In Section \ref{Small compact mu-groups}, we classify the small $\mu$-stable $\mu$-structures $(X,G)$ where $G$ acts on a compact group $X$ as homeomorphisms. In Section \ref{Examples}, we will give examples of non-small profinite structures where $nm$-independence and $\mu$-independence are different. More precisely, $\mu$-independence satisfies extension axiom but $nm$-independence does not.\\

\medskip

Throughout this paper, we write $(X,G)$ for a $G$-set $X$. For a group $H$ and subgroups $H_1,H_2$ of $H$, we write $H_1H_2:=\{h_1h_2|\ h_1\in H_1,h_2\in H_2\}$, and $H_1^{-1}:=\{h^{-1}|\ h\in H_1\}$. For $A, B\subset X$ and $a=(a_1,\ldots,a_n)\in X^n$, we write $AB$ for $A\cup B$, and $aA$ or $Aa$ for $A\cup\{a_1,\ldots,a_n\}$. Also we write $a\in X$ for a finite tuple $a$ of elements in $X$. For a $\mu$-structure $(X,G)$ and $A\subset X$, we denote $G_A$ for the pointwise sabilizer of $A$ and $\mu_A$ for the Haar measure on $G_A$. Let $\phi:\ G\times X\rightarrow X$ be the action of $G$ on $X$. For $g\in G$ and for $a\in X$ and $A\subset X$, we write $ga$ and $g[A]$ for $\phi(g,a)$ and $\phi(g,A)$ respectively. Through this paper, we consider only a Hausdorff space.

We first recall some properties of the Haar measure on locally compact groups. Let $G$ be a locally compact group. Since $G$ is locally compact, it is equipped with a Borel measure $\mu$(at least, there is the Haar measure). We recall basic notions and facts on Borel measures on locally compact spaces(c.f. \cite{DE,KT}). 
\begin{definition}\cite{DE}\label{Borel_Radon_measure}
Let $X$ be a locally compact space.
\begin{enumerate}
	\item The {\em Borel $\sigma$-algebra} $\mathfrak{B}(X)$ on $X$ is the $\sigma$-algebra generated by open subsets of $X$.
	\item A {\em Borel measure} is a countably additive non-negative measure $\mu: \mathfrak{B}(X)\rightarrow [0,+\infty]$.
\end{enumerate}
\end{definition}
\begin{definition}\cite{DE}\label{Radon_measure}
Let $X$ be a locally compact space. Let $\mu$ be a Borel measure on $X$.
\begin{enumerate}
	\item We say $\mu$ is {\em locally finite} if for any $x\in X$, there is an open neighborhood $U$ of $x$ such that $\mu(U)<\infty$.
	\item We say $\mu$ is {\em regular} if for any Borel set $E$ of $X$, 
	\begin{align*}
	\mu(E)&=\inf \{\mu(U)|\ E\subset U,\ U\mbox{ open}\}\\
	&=\sup\{\mu(K)|\ K\subset E,\ K\mbox{ compact}\}.
\end{align*}		
	\item We say $\mu$ is a {\em Rodon measure} if
	\begin{enumerate}
		\item $\mu(K)<\infty$ for any compact set $K\subset X$,
		\item $\mu(E)=\inf \{\mu(U)|\ E\subseteq U,\ U\mbox{ open}\}$ for $E\in \mathfrak{B}(X)$, and
		\item $\mu(U)=\sup\{\mu(K)|\ K\subseteq U,\ K\mbox{ compact}\}$ for $U\subseteq X$ open.
\end{enumerate}
\end{enumerate}
\end{definition}

\noindent Next we define {\em push-forward} and {\em push-out} measures. 
\begin{definition}
Let $X$ and $Y$ be locally compact spaces. We say a continuous map $f:\ X\rightarrow Y$ is {\em Borel*} if $f(A)\in \mathfrak{B}(Y)$ for each $A\in \mathfrak{B}(X)$.
\end{definition}
\noindent By the Lusin-Suslin Theorem, any injective continuous map between Polish spaces is Borel*.
\begin{theorem}[The Lusin-Suslin Theorem]\cite{Ke}\label{Lusin-Suslin}
Let $f:\ X\rightarrow Y$ be a continuous map between locally compact Polish spaces. For each $A\in \mathfrak{B}(X)$, if $f\restriction A$ is injective, then $f(A)$ is in $\mathfrak{B}(Y)$. 
\end{theorem}
\begin{remark/def}\label{push-forward_out_measure}
Let $f:\ X\rightarrow Y$ be a continuous map between locally compact spaces. Let $\mu_X$ and $\mu_Y$ be Borel measures on $X$ and $Y$ respectively.
\begin{enumerate}
	\item The composition $\mu_X\circ f^{-1}$ gives a Borel measure on $Y$. We call this measure the {\em push-forward measure} by $f$ and denote by $f_*(\mu_X)$.
	\item Suppose $f$ is Borel*. The composition $\mu_Y \circ f$ gives a Borel measure on $X$. We call this measure the {\em push-out measure} by $f$ and denote by $f^*(\mu_Y)$. 
\end{enumerate}
\end{remark/def}

\begin{remark}\cite{Bour}\label{Polish_Radon_regular}
Since a Polish space is strongly Radon, any locally finite Borel measure on a locally compact Polish space $X$ is a regular Borel measure.
\end{remark}

\begin{lemma}\label{push-out_regular_measure}
Let $f:\ X\rightarrow Y$ be a injective continuous map between locally compact Polish spaces. Let $\mu_Y$ be a locally finite Borel measure on $Y$. Then $f^*(\mu_Y)$ is a regular Borel measure on $X$.
\end{lemma}
\begin{proof}
By Remark/Definition \ref{push-forward_out_measure}, the push-out measure $f^*(\mu_Y)$ is a Borel measure. We show that the measure $f^*(\mu_Y)$ is locally finite. Take $x\in X$. Since $\mu_Y$ is locally finite, there is an open neighborhood $U$ of $f(x)$ with $\mu_Y(U)<\infty$. The open neighborgood $f^{-1}[U]$ of $x$ has the finite measure $f^*(\mu_Y)$. Therefore $f^*(\mu_Y)$ is a locally finite Borel measure on $X$ and it is regular by Remark \ref{Polish_Radon_regular}.
\end{proof}

\begin{definition}\cite{DE}\label{def_invariant}
Let $G$ be a group acting on a locally compact space $X$, and let $\mu$ be a Borel measure on $X$. For $S\subset X$ and $g\in G$, we write $gS:=\{gs|s\in X\}$ and $Sg:=\{sg|\ s\in X\}$. We say $\mu$ is {\em ($G$-)left-invariant[respectively, right-invariant]} if for any Borel set $E$ of $X$ and $g\in G$, $\mu(gS)=\mu(S)$[respectively, $\mu(Sg)=\mu(S)$]. We say $\mu$ is {\em ($G$-)invariant} if it is left and right invariant.
\end{definition}

\begin{definition}\cite{DE}\label{quasi_inv}
Let $G$ be a group acting on a locally compact space $X$. A {\em quasi-invariant measure}(with respect to $G$) $\mu$ on $X$ is a regular Borel measure such that for $g\in G$ and a Borel set $Y$ of $X$, there is a non-zero $\Delta(g)$ such that $\mu(gX)=\Delta(g)\mu(X)$.
\end{definition}

\begin{remark/def}\cite[Theorem 1.3.5]{DE}\label{Haar_measure}
Let $G$ be a locally compact group. Then there is a unique(up to scaling) non-zero left-invariant regular Radon measure $\mu$. We call this measure the (left-invariant) {\em the Haar measure} on $G$. So there is a quasi-invariant measure on a locally compact group.
\end{remark/def}

\begin{fact}\cite[Corollary 1.23]{KT}\label{quasi_inv_and_equivalence}
Let $G$ be a locally compact group and let $H$ be a closed subgroup of $G$. Let $\nu_1$ and $\nu_2$ be two quasi-invariant measures on $G/H$. Then $\nu_1$ and $\nu_2$ are equivalent, that is, they have the same null sets.
\end{fact}

\begin{theorem}[Weil's integration formula]\cite[Corollary 1.21]{KT}\label{Weil_integration_formula}
Let $H$ be a closed subgroup of a locally compact group $G$. Let $\mu_G$, $\mu_H$ be the Haar measures on $G$ and $H$ respectively. Then there are a quasi-invariant regular Borel measure $\mu$ on $G/H$ and  a continuous strictly positive function $\rho$ on $G$ such that for any integrable function $f$ on $G$,
$$\int_G f(x)\rho(x)d\mu_G(x)=\int_{G/H}\int_H f(xh)d\mu_H(h)d\mu(xH).$$
\end{theorem}

\begin{lemma}\label{WIF_char}
Let $G$ be a locally compact group, which is $\sigma$-compact(e.g.  a compact group or a locally compact Polish group), that is, $G$ is a countable union of compact sets. Let $H$ be a closed subgroup of $G$. For any non-negative measurable function $f$ on $G$, we have $$\int_G f(x)\rho(x)d\mu_G(x)=\int_{G/H}\int_H f(xh)d\mu_H(h)d\mu(xH).$$
\end{lemma}
\begin{proof}
Suppose a locally compact group  $G$ is $\sigma$-compact so that there is an increasing sequence $(X_i)_{i\in \omega}$ of compact subsets of $G$ such that $\bigcup X_i=G$. Let $H$ be a closed subgroup of $G$ and let $f$ be a positive measurable function on $G$. Let $Y_i=f^{-1}[[0,i]]$ for each $i$. Consider $f_i=f\chi_{X_i\cap Y_i}$ for each $i$, where $\chi_{X_i\cap Y_i}$ is the characteristic functions of $X_i\cap Y_i$, so that $\lim_i f_i =f$. By the Monotone Convergence Theorem and Theorem \ref{Weil_integration_formula}, we have that
\begin{align*}
\int_G f(x)\rho(x) d\mu_G(x)&=\lim_i \int_G f_i(x)\rho(x) d\mu_G(x)\\
&=\lim_i \int_{G/H}\int_H f_i(xh)d\mu_H(h)d\mu(xH)\\
&=\int_{G/H} \lim_i \int_H f_i(xh)d\mu_H(h)d\mu(xH)\\
&=\int_{G/H} \int_H \lim_i  f_i(xh)d\mu_H(h)d\mu(xH)\\
&=\int_{G/H}\int_H f(xh)d\mu_H(h)d\mu(xH).
\end{align*}
\end{proof}

\begin{notation}
Let $G$ be a group acting on a locally compact space $X$. Let $A$ be a Borel set of $X$. For a Borel measure $\nu$ on $X$, we write $A\subset_{\nu} X$ if $\nu(A)>0$. We write $A\subset_{\mu}X$ if $A\subset_{\nu}X$ for any quasi-invariant measure $\nu$ on $X$.
\end{notation}
\begin{corollary}\label{WIF_quotient}
Let $G$ be a locally compact group, which is $\sigma$-compact and let $H$ be a closed subgroup of $G$. Let $\pi:\ G\rightarrow G/H$ be the canonical projection. For a Borel set $A$ of $G/H$, we have that $$\pi^{-1}[A]\subset_{\mu} G\Leftrightarrow A\subset_{\mu}G/H.$$
\end{corollary}
\begin{proof}
Let $\mu_G$ be the Haar measure on $G$ and let $\mu$ be a quasi-invariant measure in Theorem \ref{Weil_integration_formula}. Let $A$ be a Borel subset of $G/H$. Let $A'=\pi^{-1}[A]$. By Corollary \ref{WIF_char}, we have that $$\int_G \chi_{A'}(x)\rho(x)d\mu_G(x)=\int_{G/H}\int_H \chi_{A'}(xh)d\mu_H(h)d\mu(xH).$$ Suppose $\mu_G(A')>0$ and $\mu(A)=0$. The condition $\mu_G(A')>0$ implies $\int_G \chi_{A'}(x)\rho(x)d\mu_G(x)>0$. Since $\mu(A)=0$, we have that $$\int_{G/H}\int_H \chi_{A'}(xh)d\mu_H(h)d\mu(xH)=\int_{G/H\setminus A}\int_H \chi_{A'}(xh)d\mu_H(h)d\mu(xH).$$ For $x\in G\setminus A'$ and $h\in H$, $xh\notin A'$ and $\chi_{A'}(xh)=0$. Thus we have $$\int_{G/H\setminus A}\int_H \chi_{A'}(xh)d\mu_H(h)d\mu(xH)=0,$$ which contradicts. So $\mu_G(A')>0$ implies $\mu(A)>0$. Conversely, if $\mu(A)>0$, then $\int_{G/H}\int_H \chi_{A'}(xh)d\mu_H(h)d\mu(xH)>0$ and $\mu_G(A')>0$.
\end{proof}

Next we introduce some notations for basic model theory for $\mu$-structures(c.f. \cite[Section 3]{Kru5}). We recall notions of {\em definable sets, imaginary sorts, names} and so on in \cite[Section 3]{Kru5}. Most of all, imaginary sorts are crucial to handle quotient objects in Section \ref{Small compact mu-groups}. Let $(X,G)$ be a $\mu$-structure. For $Y\subset X^n$, we define $\stab(Y):=\{g\in G|\ g[y]=Y\}$. We say that $Y$ is invariant [over a finite set $A$] if $\stab(Y)=G$[$\supset G_A$, respectively]. We recall the definitions of imaginary sorts and definable sets.
\begin{definition}\cite[Definition 3.1, 3.3]{Kru5}\label{imaginary_definable_name}
\begin{enumerate}
	\item The {\em imaginary extension}, denoted by $X^{eq}$, is the union of all sets of the form $X^n/E$ with $E$ ranging over all invariant equivalence relations such that for all $a\in X^n$, $\stab([a]_E)$ is a closed subgroup of $G$. The sets $X^n/E$ is called the (imaginary) sorts of $X^{eq}$.
	\item A subset $D$ of a sort of $X^{eq}$ is {\em definable} over a finite subset $A$ of $X^{eq}$(in short $A$-definable) if $D$ is invariant over $A$ and $\stab(D)$ is a closed subset of $G$. We say that $D$ is definable if it is definable over some $A$.
	\item We say $d\in X^{eq}$ is a name for a definable set $D$ if $\stab(D)=G_d$.
\end{enumerate}
\end{definition} 

\begin{fact}\cite[Remark 3.2, Proposition 3.4]{Kru5}
\begin{enumerate}
	\item $(X^{eq})^{eq}=X^{eq}$.
	\item Each definable set in $X^{eq}$ has a name in $X^{eq}$.
	\item Let $X^n/E$ be a sort of $X^{eq}$. Then $G$ induces a permutation group of $X^n/E$, denoted by $G\restriction X^n/E$, which is a compact group, and $(X^n/E,G\restriction X^n/E)$ is a $\mu$-structure.
\end{enumerate}
\end{fact}
\noindent We define a notion of $\mcl^{eq}$ and $\dcl^{eq}$ in $X^{eq}$ in the same way as in $X$(see Definition \ref{acl_dcl} and Definition \ref{mcl} for the notions of $\dcl$ and $mcl$ respectively).

We recall a notion of (topological) $G$-space.
\begin{definition}\cite[Definition 3.5]{Kru5}\label{G-space}
Let $G$ be a Polish group.
\begin{enumerate}
	\item A {\em (topological) G-space} is a Polish structure $(X,G)$ such that $X$ is a topological space and the action of $G$ is continuous.
	\item A {\em Polish[compact] G-space} is a $G$-space $(X,G)$ where $X$ is a Polish [compact] space.
\end{enumerate}
\end{definition}
\noindent We define $\mu$-spaces as a counterpart of Definition \ref{G-space}.
\begin{definition}
A {\em (topological) $\mu$-space} is a $\mu$-structure $(X,G)$, where $X$ is a Hausdorff space and $G$ acts continuously on $X$.
\end{definition}
\noindent Next we recall a notion of $*$-closed sets in $\mu$-spaces.
\begin{definition}\cite[Definition 3.6]{Kru5}\label{*-closed}
Assume $(X,G)$ is a $\mu$-space. We equip the quotient topology on each sort $X^n/E$ of $X^{eq}$. We say $D\subset X^n/E$ is {\em $A$-closed} for a finite $A\subset X^{eq}$ if it is closed and invariant over $A$. We say that $D\subset X^n/E$ is {\em $*$-closed} if it is $A$-closed for some finite $A$. 
\end{definition}

\begin{fact}\cite[Proposition 3.7]{Kru5}\label{topological_sort}
Let $(X,G)$ be a compact $\mu$-space and $E$ be a $\emptyset$-closed equivalence relation on $X^n$. Then $X^n/E$ is compact, and $(X^n/E,G/G_{X^n/E}$ is a compact $\mu$-space, where $G_{X^n/E}=\bigcap_{[a]_E\in X^n/E} \stab([a]_E)$.
\end{fact}
\noindent Using Fact \ref{topological_sort}, we introduce topological sorts.
\begin{definition}\cite[Definition 3.8]{Kru5}\label{teq}
Let $(X,G)$ be a compact $\mu$-space. We define $X^{teq}$, called {\em topological imaginary extension} as the disjoint union of the space $X^n/E$ with $E$ ranging over all $\emptyset$-closed equivalence relation on $X^n$. Each $X^n/E$ is called a {\em topological sort} of $X^{teq}$. We say that $D$ is $A$-closed in $X^{teq}$ if it is $A$-closed in a sort of $X^{teq}$.
\end{definition}

\noindent From Proposition \ref{equiv_mu_indo}(3), we have the following result.
\begin{fact}\cite[Remark 3.10]{Kru5}\label{sort_mu_indo}
\begin{enumerate}
	\item Let $(X,G)$ be a $\mu$-structure and $D$ an $A$-definable[or only invariant over $A$] subset of $X^{eq}$. Then $(D,G_A/G_{AD})$ is a $\mu$-structure. Moreover, if $(X,G)$ is small, then $(D,G_A/G_{AD})$ is also small. For tuples and subsets of $D$, the computation of $\mu$-independence in $(X,G_A)$ coincides with the computation of $\mu$-independence in $(D,G_A/G_{AD})$.
	\item Let $(X,G)$ be a $\mu$-space. If $D$ is an $A$-closed of $X^n$[or $X^{teq}$ if $X$ is compact], then $(D,G_A/G_{AD})$ is a $\mu$-space.
\end{enumerate}
\end{fact}
\section{mu-independence}\label{section mu-independence}
In this section, we define $\mu$-independence for $\mu$-structures and show that it satisfies the basic $4$ axioms.
\begin{definition}\label{mu-independence}
Let $(X,G)$ be a $\mu$-structure and let $\mu$ be the Haar measure on $G$. For $a\in X$ and finite sets $A,B\subset X$, we say $a$ is {\em $\mu$-independent} from $B$ over $A$, denoted by $a\indo^{\mu}_A B$, if $$\mu_A(\pi_A^{-1}[\o(a/AB)])>0.$$ Note that $\pi_A^{-1}[\o(a/AB)]=G_{AB}G_{Aa}$.
\end{definition}
\noindent Note that our $\mu$-independence generalizes $\mu$-independence for profinite structures concerned in \cite{N2}(see Remark \ref{rem:cpt_str_equiv_mu_nm_independence}).

\begin{remark}\label{push-forward_regular}
Let $H$ be a closed subgroup of a compact group $G$. Let $\mu$ be the Haar measure on $G$. Let $\pi:\ G\rightarrow G/H$ be the projection map. Then the push-forward $\pi_*(\mu)$ is a quasi-invariant measure. 
\end{remark}
\begin{proof}
It is enough to show that $\pi_*(\mu)$ is a regular Borel measure. Since $G/H$ are equipped with the quotient topology, $G$ and $G/H$ are compact, and $\pi$ is continuous, $\pi$ is closed and open map. From this, $\pi_*(\mu)$ is a regular Borel measure.
\end{proof}

\noindent For a $G$-space $X$, we consider a $G$-space $X^n$ for each $n\ge 1$ as follows: For $n\ge 1$, $x=(x_1,\ldots,x_n)\in X^n$, and $g\in G$, $gx:=(gx_1,\ldots,gx_n)$.

\begin{remark}\cite[Lemma 2.6]{Kru5}\label{Borel_H_1H_2}
Let $G$ be a Polish group. For any closed subgroups $H_1, H_2$ of $G$, $H_1 H_2$ is a Borel subset of $G$. 
\end{remark}
\begin{proposition}\label{equiv_mu_indo}
The followings are equivalent:
\begin{enumerate}
\item $a\indo^{\mu}_A B$;
\item $G_{AB} G_{Aa}\subset_{\mu}G_A$.
\item $G_{AB} G_{Aa}/G_{Aa}\subset_{\mu} G_A/G_{Aa}$
\end{enumerate}
\end{proposition}
\begin{proof}
The direction $(1)\Leftrightarrow (2)$ comes from Definition \ref{mu-independence}. The direction $(2)\Leftrightarrow (3)$ comes from Corollary \ref{WIF_quotient}.
\end{proof}
\noindent We recall several closure operations defined in \cite{Kru5} and introduce a new closure operation. 
\begin{definition}\cite[Definition 2.4]{Kru5}\label{acl_dcl}
For $A\subset X$,
\begin{itemize}
	\item $\Acl(A):=\{x\in X|\ |\o(x/A)|\le \omega\}$;
	\item $\acl(A):=\{x\in X|\ |\o(x/A)|<\omega \}$; and
	\item $\dcl(A):=\{x\in X|\ |\o(x/A)|=1  \}$.
\end{itemize}
\end{definition}
\begin{definition}\label{mcl}
For $A\subset X$, let $\mcl(A):=\{x\in X|\ G_{Ax}\subset_{\mu} G_A\}$.
\end{definition}
\begin{remark}\label{relation_between_closures}
Let $(X,G)$ be a $\mu$-structure. For $A\subset X$ finite, $\acl(A)\subset \Acl(A)\subset \mcl(A)$. If $G$ is compact, then $\acl(A)=\Acl(A)=\mcl(A)$
\end{remark}
\begin{proof}
Because of the $\sigma$-additivity, $\Acl(A)\subset \mcl(A)$. If $G$ is compact, then  $x\in \mcl(A)$ if and only if $[G_A:G_{Ax}]$ is finite. So $\mcl(A)\subset \acl(A)$.
\end{proof}

Next we show that $\mu$-independence satisfies the basic $4$ axioms.
\begin{theorem}\label{properties_mu-independence}
Let $(X,G)$ be a $\mu$-structure. Let $a,b\in X$ and $A,B,C\subset X$.
\begin{enumerate}
	\item (Invariance) For all $g\in G$, $a\indo^{\mu}_A B\Leftrightarrow ga\indo^{\mu}_{g[A]}g[B]$.
	\item (Symmetry) $a\indo^{\mu}_C b \Leftrightarrow b\indo^{\mu}_C a$.
	\item (Transitivity) Suppose $G$ is compact. Assume $A\subset B\subset C$. Then $a\indo^{\mu}_A B$ and $a\indo^{\mu}_B C$ if and only if $a\indo^{\mu}_A C$.
	\item (Extension) Suppose $(X,G)$ is small and $A,B$ are finite. There is $a'\in \o(a/A)$ such that $a'\indo^{\mu}_A B$.
	\item $a\in \mcl(A)$ if and only if for all finite $D\subset X$, $a\indo^{\mu}_A D$.
\end{enumerate}
\end{theorem}
\begin{proof}
It is easy to prove $(1),(2),(4)$ and $(5)$. Specially, $(2)$ comes from the quasi-invariance and $(4)$ comes from the $\sigma$-subadditivity.
\end{proof}

\subsection{Proof of transitivity}
We now prove transitivity. Fix a $\mu$-structure $(X,G)$.
\begin{lemma}\label{key lemma}
Suppose $H$ is a compact group or a locally compact Polish group. Let $H_1,H_2$ be closed subgroups of $H$ such that $\mu_H(H_1 H_2)>0$. Let $H_3=H_1\cap H_2$. Let $A_i\subset H_i$ be a closed subset for $i=1,2$. Suppose $A_1=A_1 H_3$ and $A_2=H_3 A_2$. Then $\mu_{H_1}(A_1)>0$ and $\mu_{H_2}(A_2)>0$ if and only if $\mu_H(A_1A_2)>0$.
\end{lemma}
\begin{proof}
Consider $H_3$ as a closed subgroup of $H_1\times H_2$ induced from the injection $\iota :\ H_3\rightarrow H_1\times H_2, h\mapsto (h,h)$. Define the following projection $\delta: H_1\times H_2\rightarrow H_1H_2,(h_1,h_2)\mapsto h_1h_2^{-1}$. Since the kernel of $\delta$ is $H_3$, we have the following diagram:
$$\begin{tikzcd}
H_1\times H_2 \ar{r}{\pi} \ar[dr,"\delta"'] & (H_1\times H_2)/H_3\ar{d}{\ov{\delta}}\\
& H_1H_2\\
\end{tikzcd}$$

\noindent Since $\mu_H(H_1 H_2)>0$, the push-out measure $\ov{\delta}^*(\mu_H)$ is a quasi-invariant Borel measure on $(H_1\times H_2)/H_3$. Take $A_1=A_1H_3$ and $A_2=H_3A_2$ so that $\ov{ \pi}(A_1A_2)=A_1A_2/H_3$.
\begin{claim}\label{A1A2=A1timesA2}
$A_1\times (A_2)^{-1}=\delta^{-1}(A_1A_2)(=\pi^{-1}\circ \ov{ \delta}^{-1}(A_1A_2))$.
\end{claim}
\begin{proof}
Let $(h_1,h_2)\in H_1\times H_2$ be such that $h_1h_2^{-1}=a_1a_2\in A_1A_2$ for some $a_1\in A_1$ and $a_2\in A_2$. We have $h:=a_1^{-1}h_1=a_2h_2\in H_1\cap H_2=H_3$. Then $h_1=a_1h\in A_1H_3$ and $h_2=a_2^{-1}h\in A_2^{-1}H_3=(H_3A_2)^{-1}=A_2^{-1}$.
\end{proof}
\noindent By Fact \ref{quasi_inv_and_equivalence}, Corollary \ref{WIF_quotient}, and Claim \ref{A1A2=A1timesA2}, we have that
\begin{align*}
\mu_{H_1}(A_1)>0,\mu_{H_2}(A_2)>0&\Leftrightarrow (\mu_{H_1}\times\mu_{H_2})(A_1\times (A_2)^{-1})>0\\
&\Leftrightarrow A_1\times (A_2)^{-1}/H_3\subset_{\mu} H_1\times H_2/H_3\\
&\Leftrightarrow \ov{\delta}^*(\mu_H)(A_1\times (A_2)^{-1}/H_3)>0\\
&\Leftrightarrow \mu_H(A_1A_2)>0,
\end{align*} 
where $\mu_{H_1}\times\mu_{H_2}$ is the product measure of $\mu_{H_1}$ and $\mu_{H_2}$.
\end{proof}

\begin{corollary}\label{cor 1}
Let $H$ be a locally compact Polish group or a compact group, and let $H_1,H_2$ be closed subgroups of $H$ with $\mu_H(H_1H_2)>0$. For any $A_1\subset H_1$ with $\mu_{H_1}(A_1)>0$, then $\mu_H(A_1H_2)>0$.
\end{corollary}

\begin{theorem}\label{transitivity}
Suppose $G$ is a locally compact Polish group or a compact group. Let $a\in X$ and let $A\subset B\subset C$. Then $a\indo^{\mu}_A B$ and $a\indo^{\mu}_B C$ if and only if $a\indo^{\mu}_A C$.
\end{theorem}
\begin{proof}
$(\Rightarrow)$ Suppose $a\indo^{\mu}_A B$ and $a\indo^{\mu}_B C$. By Proposition \ref{equiv_mu_indo}, we have $G_B G_{Aa}\subset_{\mu} G_A$ and $G_C G_{Ba}\subset_{\mu} G_B$. 
\begin{claim}
$G_C G_{Aa}\subset_{\mu }G_A$.
\end{claim}
\begin{proof}
We have $G_{Ca}=G_C G_{Aa}$. Use Corollary \ref{cor 1} for $H=G_A$, $H_1=G_B$, $H_2=G_{Aa}$, $A_1=G_C G_{Ba}$ and we have $A_1H_2=G_C G_{Ba} G_{Aa}=G_C G_{Aa}\subset_{\mu} G_A$.
\end{proof}
\medskip

$(\Leftarrow)$ Suppose $G_C G_{Aa}\subset_{\mu} G_A$. Since $G_C G_{Aa}\subset G_B G_{Aa}\subset G_A$, we have $G_B G_{Aa}\subset_{\mu} G_A$. It remains to show that $G_C G_{Ba}\subset_{\mu }G_B$.
\begin{claim}
$G_C G_{Ba}\subset_{\mu }G_B$.
\end{claim}
\begin{proof}
We have the followings
\begin{itemize}
	\item $G_B G_{Aa}\subset_{\mu} G_A$.
	\item $G_C\subset G_B$.
	\item $G_{Ba}=G_B\cap G_{Aa}$.
\end{itemize}
Consider the following map $\delta:G_B\times G_{Aa}\rightarrow G_B G_{Aa}, (x,y)\mapsto xy^{-1}$ and it induces the following digram:
$$\begin{tikzcd}
G_B\times G_{Aa} \ar{r}{\pi} \ar[rd,"\delta"'] & (G_B\times G_{Aa})/G_{Ba} \ar{d}{\ov{\delta}}\\
& G_B G_{Aa}
\end{tikzcd}$$

\begin{claim}\label{G_CG_BG_AaG_cG_Aa}
$\delta^{-1}(G_C G_{Aa})=G_C G_{Ba}\times G_{Aa}$
\end{claim}
\begin{proof}
It is enough to show that $\pi^{-1}(G_C G_{Aa})\subset G_C G_{Ba}\times G_{Aa}(*)$. Choose $x\in G_B$ and $y\in G_{Aa}$ such that $xy^{-1}\in G_C G_{Aa}$. Then $x\in G_C G_{Aa}y=G_C G_{Aa}$ and $x\in G_C G_{Aa}\cap G_B$. We show that $G_C G_{Aa}\cap G_B=G_C G_{Ba}$. It is enough to show that $G_C G_{Aa}\cap G_B\subset G_C G_{Ba}$. Let $z\in G_C G_{Aa}\cap G_B$ and let $u\in G_C,\ v\in G_{Aa}$ be such that $uv=z$. Then $v=u^{-1}z\in G_C G_B =G_B$ and $v\in G_{Aa}\cap G_B=G_{Ba}$. Therefore $z\in uG_{Ba}\subset G_C G_{Ba}$. So $x\in G_C G_{Ba}$ and $(*)$ holds.
\end{proof}
\noindent Since $G_B G_{Aa}\subset_{\mu}G_A$, the push-out measure $\ov{\delta}^*(\mu_{G_A})$ is a quasi-invariant Borel measure on $(G_B\times G_{Aa})/G_{Ba}$. By Fact \ref{quasi_inv_and_equivalence}, Corollary \ref{WIF_quotient}, and Claim \ref{G_CG_BG_AaG_cG_Aa}, we have that
\begin{align*}
G_CG_{Aa}\subset_{\mu} G_A&\Leftrightarrow \ov{\delta}^*(\mu_{G_A})((G_C G_{Ba}\times G_{Aa})/G_{Ba})>0\\
&\Leftrightarrow (G_C G_{Ba}\times G_{Aa})/G_{Ba}\subset_{\mu} (G_B\times G_{Aa})/G_{Ba}\\
&\Leftrightarrow G_C G_{Ba}\times G_{Aa}\subset_{\mu} G_B\times G_{Aa}.
\end{align*}
Since $G_CG_{Aa}\subset_{\mu} G_A$, we have that $G_C G_{Ba}\times G_{Aa}\subset_{\mu} G_B\times G_{Aa}$, and $G_C G_{Ba}\subset_{\mu} G_B$. 
\end{proof}
\end{proof}

\subsection{Some description of mu-independence}
In this subsection, we describe $\mu$-independence intrinsically on $X$ for a $\mu$-structure $(X,G)$(see Theorem \ref{description_mu_independence}). This is crucial to classify small compact $\mu$-groups in Section \ref{Small compact mu-groups}(see Proposition \ref{equvialence_mu_stability}, Remark \ref{mu_rank_measure_complexity}, and Proposition \ref{existence_mu_generic}). Krupinski in \cite{Kru5} described $nm$-independence in $G$-spaces as follows:
\begin{fact}\label{fact:description_nm_indep}\cite[Theorem 2.12]{Kru5}
Let $(X,G)$ be a $G$-space. Let $a,A,B\subset X$ be finite sets. Suppose $\o(a/A)$ is non-meager in its relative topology. Then $a\indo^{nm}_A B$ if and only if $\o(a/AB)$ is non-meager in $\o(a/A)$.  
\end{fact}

\begin{theorem}\label{description_mu_independence}
Let $(X,G)$ be a $\mu$-space. Let $a,A,B\subset X$ be finite. Suppose
\begin{enumerate}
	\item The canonical map $\ov{\pi}_{A,a}:\ G_A/G_{Aa}\rightarrow \o(a/A),g\mapsto ga$ is Borel*.
	\item There is a locally finite quasi-invariant Borel measure $\nu$ on $(\o(a/A),G_A)$.
\end{enumerate}
Then, $$a \indo^{\mu}_A B\Leftrightarrow\o(a/AB)\subset_{\nu}\o(a/A).$$
\end{theorem}
\begin{proof}
For the simplicity, we assume that $A=\emptyset$. Suppose the canonical map $\pi_a:\ G\rightarrow \o(a),g\mapsto ga$ is Borel* and there is a locally finite quasi-invariant Borel measure $\nu$ on $(\o(a),G)$.  Then $\pi$ factors through in the following way:
$$\begin{tikzcd}
G \ar{r}{\pi} \ar[rd,"\pi_a"']& G/G_{a}\ar{d}{\ov{\pi}_a}\\
&\o(a)
\end{tikzcd}$$
Then the push-out measure $\ov {\pi}_a^*(\nu)$ is a quasi-invariant regular Borel measure on $G/G_a$. Then we have that
\begin{align*}
a\indo^{\mu} B&\Leftarrow G_B G_a/G_a\subset_{\mu} G/G_a\\
&\Leftrightarrow \ov {\pi}_a^*(\nu)(G_B G_a/G_a)>0\\
&\Leftrightarrow \nu(\o(a/B))>0\\
&\Leftrightarrow \o(a/B)\subset_{\nu}\o(a).
\end{align*}
\end{proof}
\begin{remark}\label{description_mu_independence_in_some_exmaples}
\begin{enumerate}
	\item If $G$ is a locally compact Polish group and $X$ is a Polish space, then the condition (1) in Theorem \ref{description_mu_independence} holds.
	\item If $G$ is a compact group, then both conditions (1) and (2) in Theorem \ref{description_mu_independence} hold.
\end{enumerate}
\end{remark}
\begin{proof}
(1) It comes from the Lusin-Suslin Theorem.

(2) It comes from the fact that the map $\ov{\pi}_a:\ G/G_{Aa}\rightarrow \o(a/A)$ is a homeomorphism if $G$ is a compact group.
\end{proof}
\begin{remark}
Let $(X,G)$ be a $\mu$-structure. Let $a,A,B\subset X$ be finite. Then there is the canonical map $\ov{\pi}_{A,a}:\ G/G_{Aa}\rightarrow \o(a/A)$. Give a topology on $\o(a/A)$ induced by $\ov{\pi}_{A,a}$, that is, $U\subset \o(a/A)$ is open if and only if $\ov{\pi}_{A,a}^{-1}(U)$ is open. With this topology, there is a quasi-invariant Borel measure $\nu$ on $\o(a/A)$ such that $$a\indo^{\mu}_A B\Leftrightarrow \o(a/AB)\subset_{\nu}\o(a/A).$$ In this setting, our $\mu$-independence coincides exactly with $\mu$-independence of Newelski in the case of profinite structures(c.f. \cite{N1,N2}).
\end{remark}

\begin{definition}\cite{Kru4,Kru5}\label{cptstr}
Let $(X,G)$ be a Polish structure.
\begin{enumerate}
	\item We say $(X,G)$ is {\em a profinite structure} if $X$ is a profinite metric space and $G$ is a profinite group continuously acting on $X$.
	\item We say $(X,G)$ is {\em a compact structure} if $X$ is a compact metric space and $G$ is a compact group continuously acting on $X$.

\end{enumerate} 
\end{definition}

\begin{definition}\cite[Definition 2.12]{Kru5}
Let $(X,G)$ be a compact structure. For finite subsets $a,A,B$ of $X$, we say $a$ is {\em $m$-independent} from $B$ over $A$ if $\o(a/AB)$ is open in $\o(a/A)$, denoted by $a\indo^{m}_A B$.  
\end{definition}
\noindent Krupinski in \cite{Kru5} showed that two independence notions $\indo^{nm}$ and $\indo^m$ coincide in compact structures.
\begin{question}\label{question_equivalence_meager_mu}
For $\mu$-spaces, we can define $nm$-independence. In this case, is $\mu$-independence coincident with $nm$-independence?
\end{question}
\noindent Clearly, $nm$-independence implies $\mu$-independence. In Section \ref{Examples}, we give an example of $\mu$-space $(X,G)$ such that $\mu$-independence and $nm$-independence are different, where $(X,G)$ is a profinite sturcture. In Remark \ref{rem:cpt_str_equiv_mu_nm_independence}, we show that $\indo^{nm}=\indo^{\mu}$ in small, $nm$-stable compact structures.

\section{mu-stability}\label{Basic model theory for mu-structure}
In this section, we define a rank, called the {\em $\mu$-rank}, coming from $\mu$-independence and using the $\mu$-rank, we define $\mu$-stability for $\mu$-structures. We omit the proof in this section because the proof are exactly same with ones in \cite[Section 3]{Kru5}. For a given well-behaved independence relation(enough to satisfying $(1)-(4)$ in Theorem\ref{mu-independence}, we define a notion of rank having nice properties, for example, {\em Lascar inequality}. For example, in stable theory(also in simple theory), we define the $U$-rank from forking independence, and more generally, in rosy theory, we definite the thorn $U$-rank form thorn forking independence.
\begin{definition}\label{mu-rank}
Let $(X,G)$ be a $\mu$-structure. The {\em $\mu$-rank}, denoted by $\mR$, is the unique ordinal-valued function from the collection of orbits over finite sets satisfying $\mR(a/A)\ge \alpha+1$ iff there is a finite set $B\supset A$ such that $a\depo^{\mu}_A B$ and $\mR(a/B)\ge \alpha$. We define $\mR(a/A)=\infty$ if for any ordinal $\alpha$, $\mR(a/A)\ge \alpha$.

Let $D\subset X^{eq}$ be definable over $A$. Define $\mR(D):=\sup \{\mR(d/A)|\ d\in D\}$.
\end{definition}

\begin{fact}\cite[Remark 5.10, Remark 5.11]{Kru5}\label{fact:stabilizer_clopensets_relativised_mu_independence}
Let $(X,G)$ be a compact $\mu$-structure. Let $G_0$ be a closed subgroup of $G$ having countable index.
\begin{enumerate}
	\item Let $U$ be a clopen subset of $X$. Then $\stab(U)$ is a clopen subgroup of $G$, and so $[G:\stab(U)]\le \omega$.
	\item Suppose $(X,G)$ is small. Then $(X,G_0)$ is small, and $\indo^{\mu}$ and $\mu$-rank computed in $(X,G)$ are the same as in $(X,G_0)$.
\end{enumerate} 
\end{fact}
\begin{proof}
(1) By the same reasoning in the proof of \cite[Remark 5.10]{Kru5}, we have that $\stab(U)$ is an open subgroup of $G$. For a $\sigma$-compact group(e.g. a compact group or a Polish locally compact group), any open subgroup has countable index. So we have that $[G:\stab(U)]\le \omega$.\\

(2) By the same reasoning in the proof of \cite[Remark 5.11]{Kru5}, it is enough to check that if $a\indo^{\mu}_A B$ in $(X,G_0)$, then $a\indo^{\mu}_A B$ in $(X,G)$. Let $\nu$ be the Haar measure on $G$. Since $G_0$ is a closed subgroup of countable index, we have $\nu(G_0)>0$. From that $a\indo^{\mu}_A B$ in $(X,G_0)$, we have $G_{0AB}G_{0Aa}\subset_{\mu}G_{0A}$ and $\nu(G_{0AB}G_{0Aa})>0$. Since $G_{0AB}G_{0Aa}\subset G_{AB}G_{Aa}$, we have $\nu(G_{AB}G_{Aa})>0$ and $G_{AB}G_{Aa}\subset_{\mu}G$, that is, $a\indo^{\mu}_A B$ in $(X,G)$.
\end{proof}

We list useful properties of the $\mu$-rank coming from the standard forking calculation and transfinite induction(see \cite{Kim,Wagner}).
\begin{proposition}[Lascar inequalities for the $\mu$-rank]\label{Lascar inequality}
Let $a,b,A\subset X^{eq}$ be a finite subsets.
\begin{enumerate}
	\item $a\indo^{\mu}_A b$ implies $\mR(a/Ab)=\mR(a/A)$. The converse holds if $\mR(a/A)<\infty$.
	\item $\mR(a/Ab)+\mR(b/A)\le \mR(ab/A)\le \mR(a/Ab)\oplus)\mR(b/A)$.
	\item Suppose that $\mR(a/Ab)<\infty$ and $\mR(a/A)\ge \mR(a/Ab)\oplus\alpha$. Then, $\mR(b/A)\ge \mR(b/Aa)+\alpha$.
	\item Suppose that $\mR(a/Ab)<\infty$ and $\mR(a/A)\ge \mR(a/Ab)+\omega^{\alpha}n$. Then, $\mR(b/A)\ge \mR(b/Aa)+\omega^{\alpha}n$.
	\item If $a\indo^{\mu}_A b$, then $\mR(ab/A)=\mR(a/Ab)\oplus \mR(b/A)$.
\end{enumerate}
\end{proposition}
\begin{remark}\cite[Remark 3.13]{Kru5}\label{intrinsity_mu_rnak}
Let $a,A\subset X$ be finite. Then the computation of $\mR(a/A)$ in $X$ is the same as the computation of $\mR(a/A)$ in $X^{eq}$.
\end{remark}
Next we introduce $\mu$-stability analogous to $nm$-stability of Polish structure.
\begin{definition}\label{mu-stability}
$(X,G)$ is {\em $\mu$-stable} if every $1$-orbit has ordinal $\mu$-rank.
\end{definition}
\begin{proposition}\cite[Remark 3.15, 3.16, Proposition 3.17]{Kru5}\label{equvialence_mu_stability}
The following are equivalent:
\begin{enumerate}
	\item $(X,G)$ is $\mu$-stable.
	\item Each $n$-orbit, $n\ge 1$, has ordinal $\mu$-rank.
	\item Each $1$-orbit in $X^{eq}$ has ordinal $\mu$-rank.
	\item There are no finite sets $A_0\subset A_1\subset \ldots\subset X$ and $a\in X$ such that $a\depo^{\mu}_{A_i} A_{i+1}$ for every $i\in \omega$.
	\item There are no finite sets $A_0\subset A_1\subset \ldots\subset X^{eq}$ and $a\in X^{eq}$ such that $a\depo^{\mu}_{A_i} A_{i+1}$ for every $i\in \omega$.
	\item There are no finite sets $A_0\subset A_1\subset \ldots\subset X$ and $a\in X$ such that $G_{A_{i+1}}G_{A_i a}\not\subset_{\mu}G_{A_i}$ for every $i\in \omega$.
	\item For every finite sets $A_0\subset A_1\subset \ldots\subset X$ and $a\in X$, there is $n\in\omega$ such that $G_{A_{n+i+1}}G_{A_n a}\subset_{\mu} G_{A_{n+i}}G_{A_n a}$ for every $i\in \omega$.
\end{enumerate}
\end{proposition}
If $(X,G)$ is $\mu$-stable, the $\mu$-rank of an orbit is computed intrinsically in the orbit itself.
\begin{proposition}\cite[Proposition 1.1, Corollary 1.2]{N1}
Suppose $(X,G)$ is $\mu$-stable. Let $a\in X^{eq}$ and let $A$ be a finite subset of $X$. For an ordinal $\alpha$, the following are equivalent:
\begin{enumerate}
	\item $\mR(a/A)\ge \alpha+1$.
	\item There is a finite set $B\subset \o(a/A)$ such that $a\depo^{\mu}_A B$ and $\mR(a/AB)\ge \alpha$.
\end{enumerate}
\end{proposition}
\begin{corollary}\label{mu-stable_mu_rank_instrinsic}
Suppose $(X,G)$ is $\mu$-stable. Let $D$ be a definable over $A$ and $a,B\subset D$ finite. Then the computation of $\mR(a/AB)$ in $(X,G)$ is equal to the computation of $\mR(a/AB)$ in $(D,G_A/G_{AD})$.
\end{corollary}
In the case of $\mu$-space, using Theorem \ref{description_mu_independence}, we can describe $\mu$-stability in terms of $X$, and the $\mu$-rank measures `measure theoretic complexity' of orbits, which is a counterpart of \cite[Remark 3.20]{Kru5}.
\begin{remark}\label{mu_rank_measure_complexity}
Let $(X,G)$ be a $\mu$-space. Suppose for any orbit $\o(a/A)$ over a finite set $A$,
\begin{itemize}
	\item the canonical map $\ov{\pi}_{A,a}:\ G_A/G_{Aa}\rightarrow \o(a/A)$ is Borel*.
	\item there is a locally finite quasi-invariant Borel measure on $\nu_{Aa}$ on $(\o(a/A),G_A)$.
\end{itemize}
\begin{enumerate}
	\item $\mR(a/A)\ge \alpha+1$ iff there is a finite set $B\supset A$ such that $\o(a/B)\not\subset_{\nu_{Aa}}\o(a/A)$ and $\mR(a/B)\ge \alpha$.
	\item $(X,G)$ is $\mu$-stable iff there are no finite sets $A_0\subset A_1\subset \ldots\subset X$ and $a\in X$ such that $\o(a/A_{i+1)})\not\subset_{\nu_{A_ia}}\o(a/A_i)$ for every $i\in \omega$.
\end{enumerate}
\end{remark}

Now we define a pregeometry on an orbit of the $\mu$-rank $1$, which is a counter part of \cite[Proposition 3.23]{Kru5}. For a finite set $A\subset X^{eq}$, we define $\mcl^{eq}_A(B):=\mcl^{eq}(AB)$. Then $\mcl^{eq}_A$ gives a pregeometry on a obit over $A$ of $\mu$-rank $1$.
\begin{remark}\label{characterization_of_acl}
\begin{enumerate}
	\item For any finite $a,A\subset X^{eq}$, $\mR(a/A)=0$ iff $a\in \mcl^{eq}(A)$.
	\item Suppose $\mR(a/A)=1$ and $B\subset X^{eq}$ is a finite subset. Then $a\in \mcl^{eq}_A(B)$ iff $a\depo^{\mu}_A B$.
\end{enumerate}
\end{remark}
\begin{proposition}\label{mu-rank1_pregeometry}
Suppose $\mR(a/A)=1$. Then $(\o(a/A),\mcl^{eq}_A)$ is a pregeometry.
Let $D$ be definable over $A$ with $\mR(D)=1$. Then $(D,\mcl^{eq}_A)$ is a pregeometry.
\end{proposition}

\section{Small compact $\mu$-groups}\label{Small compact mu-groups}
In this section we get the same structure theorems for small compact $\mu$-groups  in \cite{Kru5, KW}. The same proofs work for the case of $\mu$-groups and we omit the detailed proofs. We first define several notions analogous to \cite[Definition 5.1, 5.2]{Kru5}.
\begin{definition}\label{mu-group}
Let $G$ be a compact group or a locally compact Polish group.
\begin{enumerate}
	\item A $\mu$-group structure is a $\mu$-structure $(H,G)$ such that $H$ is a group and $G$ acts as a group of automorphisms of $H$.
	\item A {\em (topological) $\mu$-group} is a $\mu$-group structure $(H,G)$ such that $H$ is a topological group and the action of $G$ on $H$ is continuous.
	\item A {\em (locally compact, or Polish) compact $\mu$-group} is a $\mu$-group $(H,G)$ where $H$ is a (locally compact, or Polish) compact group.
\end{enumerate}
\end{definition}

\begin{definition}\label{definable_group}
\begin{enumerate}
	\item We say that a group $H$ is {\em definable} in a $\mu$-structure $(X,G)$[or in $X^{eq}$] if $H$ and the group operation on $H$ are definable in $(X,G)$[or in $X^{eq}$].
	\item We say that a group $H$ is $*$-closed in a $\mu$-space $(X,G)$[or in $X^{teq}$, if $X$ is compact] if $H$ and the group operation on $H$ are $*$-closed.
\end{enumerate}
\end{definition}

Let $H$ be a definable group over $\emptyset$ in a small $\mu$-structure $(X,G)$[or in $X^{eq}$]. Then $(H,G/G_H)$ is a small $\mu$-group structure. For convenience, we may assume that $(X,G)=(H,G)$ is a small $\mu$-group structure. Let $a\in H$ and $A\subset X^{eq}$ be finite.
\begin{definition}\label{mu-generic}
We say that the orbit $\o(a/A)$ is {\em left $\mu$-generic}(or that $a$ is left $\mu$-generic over A) if for all $b\in H$ with $a\indo^{\mu}_A b$, one has that $b a\indo_A^{\mu} b$. We say that it is {\em right $\mu$-generic} if for $b$ as above, we have $a b\indo^{\mu}_A b$. An orbit is $\mu$-generic if it is both right and left $\mu$-generic.
\end{definition}
\begin{remark}\label{properties_generics}
\begin{enumerate}
	\item If $a$ is left (right) $\mu$-generic over $A$, then $a\indo^{\mu} A$.
	\item Being left (right) $\mu$-generic is preserved under taking restrictions and $\mu$-independent extensions.
	\item Left $\mu$-generic coincide with right $\mu$-generic and so with $\mu$-generic.
	\item In the $\mu$-stable case, being $\mu$-generic means being of maximal $\mu$-rank.
\end{enumerate}
\end{remark}

Next we show that $(H,G)$ has an $\mu$-generic.
\begin{proposition}\label{existence_mu_generic}
Let $(H,G)$ be a small locally compact $\mu$-group, or, more generally, $H$ is a group definable over $C$ in a small $\mu$-space $(X,G)$[or in $X^{eq}$] such that $(H,G_C/G_{CH})$ is a locally compact $\mu$-group. Suppose each orbit over a finite set is a Borel set of $H$, and the canonical map $\ov{\pi}_{A,a}:\ G_A/G_{Aa}\rightarrow \o(a/A)$ is Borel* for any orbit $\o(a/A)$ over a finite set $A$. Then, there is at least one $\mu$-generic orbit in $H$, and an orbit $\o$ is $\mu$-generic in $H$ if and only if $\o\subset_{\mu} H$.
\end{proposition}
\begin{proof}
Consider the case when $(H,G)$ is a small locally compact $\mu$-group. Let $\mu_H$ be the Haar measures on $G$ and $H$ respectively. Fix a finite subset $A\subset H$. Each orbit of the $G_A$-action is a Borel subset of $H$, and there is an orbit $\o=\o(h/A)$ with $\mu_H(\o)>0$ because of smallness. We show that $\o$ is $\mu$-generic. Consider any $h'\indo^{\mu}_A h$.
\begin{claim}
$\mu_H(\o(h/Ah'))>0$.
\end{claim}
\begin{proof}
Consider $\delta:=\pi_{A,h}:\ G_A\rightarrow \o, g\mapsto gh$, which is Borel*. The map $\pi$ factors through as follows:
$$\begin{tikzcd}
G_A \ar{r}{\pi} \ar[rd,"\delta"'] & G_A/G_{Ah}\ar{d}{\ov{\delta}}\\
&\o
\end{tikzcd}$$

\noindent Then the push-out measure $\ov{\delta}^*(\mu_H)$ is a quasi-invariant regular Borel measure on $G_A/G_{Ah}$. From $h\indo_A^{\mu}h$, we have that
\begin{align*}
G_{Ah'} G_{Ah}\subset_{\mu}G_A &\Rightarrow G_{Ah'} G_{Ah}/G_{Ah}\subset_{\mu}G_A/G_{Ah}\\
&\Leftrightarrow \ov{\delta}^*(\mu_H)(G_{Ah'} G_{Ah})>0\\
&\Leftrightarrow \mu_H(\o(h/Ah'))>0.
\end{align*}
\end{proof}
\noindent Since $\mu_H(\o(h/A,h'))>0$, we have that 
$$\o(h/A,h')h'\subset_{\mu}H\Rightarrow \o(hh'/A,h')\subset_{\mu} H\Rightarrow \o(hh'/A,h')\subset_{\mu_H}\o(hh').$$ By Theorem \ref{description_mu_independence}, we have that $hh'\indo^{\mu}_A h'$. We have proved that any orbit of positive $\mu_H$-measure value is $\mu$-generic and so there is an $\mu$-generic orbit exists. It remains to show the converse. It is exactly same with the proof of \cite[Proposition 5.5]{Kru5}
\end{proof}

\begin{remark}\label{haarmeasure_G-quasiinv}
Let $(H,G)$ be a locally compact $\mu$-group and let $\nu$ be the Haar measure on $H$, which is $H$-invariant. The Haar measure $\nu$ is  $G$-quasi-invariant, that is, for any $g\in G$, $\nu\circ \ov g=c(g)\nu$ for some constant $c(g)>0$, where $\ov g:\ H\rightarrow H,\ h\mapsto gh$ is a homeomoerphism of $H$. So, $(H,G)$ satisfies the assumption of Proposition \ref{existence_mu_generic} if we take $G$ compact, or both $G$ and $H$ Polish.
\end{remark}
\begin{proof}
Let $\nu$ be the Haar measure on $H$. Let $g\in G$. Define $\nu'=\nu\circ \ov g$. Since $\ov g$ is a homeomorphism and $\nu$ is $H$-invariant, $\nu'$ is also the Haar measure on $H$, that is, a $H$-invariant regular Radon measure. Therefore by Remark/Definition \ref{Haar_measure}, there is $c(g)>0$ such that $\nu'=c(g)\nu$.

If $G$ is compact, then each orbit over a finite set is a closed subset of $H$. If both $G$ and $H$ are Polish, then by the Lusin-Suslin theorem, each orbit over a finite set is a Borel subset of $H$.
\end{proof}
Using Proposition \ref{existence_mu_generic}, we get a corollary analogous to \cite[Corollary 5.6]{Kru5}.
\begin{corollary}\label{rank_between_subgroups}
Let $(H,G)$ be a small locally compact $\mu$-group, or, more generally, $H$ is a group definable over $C$ in a small $\mu$-space $(X,G)$[or in $X^{eq}$] such that $(H,G_C/G_{CH})$ is a locally compact $\mu$-group. Suppose each orbit over a finite set is a Borel set of $H$, and the canonical map $\ov{\pi}_{A,a}:\ G_A/G_{Aa}\rightarrow \o(a/A)$ is Borel* for any orbit $\o(a/A)$ over a finite set $A$. Let $H_1<H_2$ be closed subgroups of $H$ definable in $X^{eq}$.
\begin{enumerate}
	\item Let $\mu_2$ be the Haar measure on $H_2$. If $\mu_2(H_1)=0$ and $\mR(H_2)<\infty$, then $\mR(H_1)<\mR(H_2)$.
	\item If $H_1$ is a open subgroup of $H_2$, then $\mR(H_1)=\mR(H_2)$.
\end{enumerate}
\end{corollary}
\begin{proof}
In the case of $(1)$, we have that $[H_2:H_1]>\aleph_0$ by $\sigma$-additivity of the measure. In the case of $(2)$, if $G$ is compact, then $[H_2:H_1]< \aleph_0$ and if $G$ is a Polish group, then $[H_2:H_1]\le \aleph_0$. The remaining proof is exactly same with the proof of \cite[Corollary 5.6]{Kru5}.
\end{proof}
For small [compact] $\mu$-groups, we have the same result with \cite[Proposition 5.7, Corollary 5.9]{Kru5}.
\begin{proposition}\label{loc_finiteness}
\begin{enumerate}
	\item Let $(H,G)$ be a small $\mu$-group. Then any finitely generated subgroup of  $H$ is countable and does not have limits points in $H$.
	\item Let $(H,G)$ be a small compact $\mu$-group. Then $H$ is a profinite torsion group.
\end{enumerate}
\end{proposition}

\noindent From now on, we assume that $(H,G)$ is a small compact $\mu$-group. In stead of the $\NM$-rank, we use the $\mu$-rank and we get the same structure results(\cite[Theorem 5.19, 5.24]{Kru5}) of small compact $\mu$-groups. 
\begin{remark}\label{Remark:clopensubgroup}
Let $K$ be a closed subgroup of a compact group $H$. Then, $K$ is a open subgroup of $H$ if and only if $K$ has the non-empty interior if and only if  $[H:K]\le \omega$ if and only if $K\subset_{\mu}H$.
\end{remark}
\noindent We have the following result for $\mu$-groups analogous to \cite[Theorem 5.19]{Kru5}.
\begin{theorem}\label{str_small_cpt_mu_group}
Let $(H,G)$ be a small, $\mu$-stable, compact $\mu$-group, then $H$ is solvable-by-finite.
\end{theorem}
\begin{proof}
With our assumption, it is reduced to the problem whether to apply Corollary \ref{rank_between_subgroups}(1) to the {\em Frattini subgroup} $\Phi(H)$ of $G$, which is the intersection of all maximal open subgroups of $H$. In our case, $\Phi(H)$ is a closed subgroup which is not open, and so $\Phi(H)\not\subset_{\mu} H$ by Remark\ref{Remark:clopensubgroup}. The remaining proof is exactly same with the proof of \cite[Theorem 5.19]{Kru5}.
\end{proof}
\noindent Next one is the analogy of \cite[Theorem 5.24]{Kru5} for $\mu$-groups.
\begin{theorem}\label{str_small_cpt_finiterank_mu_group}
If $(H,G)$ is a small compact $\mu$-group of finite $\mu$-rank, and $H$ is solvable-by-finte, then $H$ is nilpotent-by-finite.
\end{theorem}
\noindent Combining Theorem \ref{str_small_cpt_mu_group}, and \ref{str_small_cpt_finiterank_mu_group}, we have the following result.
\begin{corollary}
If $(H,G)$ is a small compact $\mu$-group of finite $\mu$-rank, then $H$ is nilpotent-by-finite.
\end{corollary}
\noindent By the same proofs in \cite{KW} with Theorem \ref{description_mu_independence}, Fact \ref{fact:stabilizer_clopensets_relativised_mu_independence}, and Remark \ref{Remark:clopensubgroup}, we have finally the following structure theorem for small, $\mu$-stable compact $\mu$-groups.
\begin{theorem}\label{Thoerem:str_small stable compact groups}
Let $(H,G)$ be a small compact $\mu$-group.
\begin{enumerate}
	\item If $(H,G)$ is $\mu$-stable, then $H$ is nilpotent-by-finite.
	\item If $\mR(H)<\omega$ or $\mR(H)=\omega^{\alpha}$ for some ordinal $\alpha$, then $H$ is abelian-by-finite.
	\item If $H$ is countable-by-abelian-by-countable, then $H$ is abelian-by-finite.
\end{enumerate}
\end{theorem}

\section{Counter example for Question \ref{question_equivalence_meager_mu}}\label{Examples}\label{Exmaples}
We give a $\mu$-structure $(X,G)$ where $(X,G)$ is a profinite structure where $\mu$-independence and $nm$-independence are different.
\begin{example}\label{different_meager_mu_no_small}
For a $n\ge 1$, let $[n]:=\{1,2,\ldots, n\}$. Consider the symmetric group $G=S_n$ acting on the set $[n]$ for $n\ge 2$. Let $H_1^n(=:H_1)$ be the group of elements fixing $n$ and $H_2^n(=:H_2)$ be the group of elements fixing $1$. Note that $H_1H_2\neq G$ because there is no element in $H_1 H_2$ sending $1$ to $n$. Since $H_2$ contains the $(n-1)$-cycle permuting $\{2,\ldots,n\}$, $|H_1H_2|\ge (n-1)!(n-1)$ and so $|H_1H_2|$ can be arbitrary close to $1$ for large enough $n$.

Fix a sequence of rational numbers $0<x_i<1$ such that the infinite product $x_0x_1 x_2\ldots >1/2$. For each $i$, take $n_i$ large enough so that $|H_1^{n_i}H_2^{n_i}|\ge x_i$. Consider the infinite product group $G=\prod_i S_{n_i} $ acting on the set $X=\prod_i [n_i]$, which is a compact structure. Really $X$ and $G$ are countably many products of finite discrete sets and so they are compact Polish spaces. But $(X,G)$ is not small. The group $G$ has the Haar measure which is just the product of the counting measures on the groups $S_{n_i}$. Let $H_1:=\prod_i H_1^{n_i}$ and $H_2:=\prod_i H_2^{n_i}$, which are the stabilizer of $(n_i)$ and $(1)$ respectively. Then the measure of $H_1H_2$ is at least $1/2$, but $H_1H_2$ has empty interior because $G$ is equipped with the product topology and $H_1^{n_i}H_2^{n_i}\neq G$ for each $i\ge 0$. From this, we conclude that $a\indo^{\mu} b$ but $a\depo^{nm}b$ for $a=(1), b=(n_i)\in X$. Actually, $\indo^{nm}$ does not satisfy the extension axiom in this case.  Take finite sets $A,B,C\subset X$. We write $A=\{(a_{ji})|j<p\}$, $B=\{(b_{ji}, j<q)\}$, and $C=\{(c_{ji}, j<r)\}$. For each $i$, let $A(i):=\{a_{ji},j<p\}$, $B(i):=\{b_{ji},j<q\}$, $C(i):=\{c_{ji},j<r\}$. We have that $C\indo^{nm}_A B$ if and only if $\{i\ge 0|\ G_{A(i)C(i)}G_{A(i)B(i)}=G_{A(i)}\}$ is cofinite if and only if either $\{i\ge 0|\ G_{A(i)C(i)}=G_{A(i)} \}$ or $\{i\ge 0|\ G_{A(i)B(i)}=G_{A(i)} \}$ is cofinite if and only if either $C\subset \acl(A)$ or $B\subset \acl(A)$. If we take $C,B\not\subset\acl(A)$, then $\sigma C\indo^{nm}_A B$ for all $\sigma \in G$.  

Next we show that $(X,G)$ satisfies the extension axiom by choosing the sequence $(n_i)$ properly. Note that $(X,G)$ is not small. We fix an increasing sequence $(n_i)$ of positive integers such that for any integers $p,q,r\ge 0$, there is $i_0$ such that for each $i\ge i_0$ and integers $0\le p_i\le p$, $0\le q_i\le q$, $0\le r_i\le r$, the infinite product $\prod_{i\ge i_0}\limits d_i $ is non-zero, where \begin{align*}
d_i&=\frac{\prod_{0\le k<q_i}\limits(n_i-(p_i+q_i+k))}{\prod_{0\le k<r_i}\limits (n_i-(p_i+k))}\\
&=\prod_{{0\le k<r_i}}(1-\frac{q_i}{n_i-(p_i+k)}).
\end{align*}
Note that $\sum \ln(1-a_i)$ converges if and only if $\sum a_i$ converges for a sequence $(a_i)$ of non-negative real numbers. So, if we take $n_i=2^i$, then it works.

\begin{claim}\label{extension_axiom_(X,G)}
The extension axiom holds for $\mu$-independence of $(X,G)$.
\end{claim}
\begin{proof}
Take finite subsets $A,B,C\subset X$. Let $p=|A|$, $q=|B|$, and $r=|C|$. For each $i$, let $G(i):=S_{n_i}$. We will find $\sigma:=(\sigma_i)\in G$ such that $\prod_{i} D_i$ is non-zero, where for each $i$, $$D_i:=\frac{|G(i)_{A(i)B(i)}G(i)_{A(i)\sigma_i[C(i)]}|}{|G(i)_{A(i)}|}.$$

Take $i_0$ such that $p+q+r<n_{i_0}$. WLOG we may assume that for all $i$, $p+q+r<n_i$. Fix $i$. Let $p_i=|A(i)|$, $q_i=|B(i)\setminus A(i)|$, and $r_i=|C(i)\setminus A(i)|$. If $r_i=0$, then $D_i=1$ and take $\sigma_i=\id\in G(i)_{A(i)}$. If $r_i>0$, we take $\sigma_i \in G(i)_{A(i)}$ such that $r_i=|\sigma_i[C(i)]\setminus A(i)B(i)|$. Then we have that
\begin{itemize}
	\item $|G(i)_{A(i)}|=(n_i-p_i)!$,
	\item $|G(i)_{A(i)B(i)}|=(n_i-(p_i+q_i))!$,
	\item $|G(i)_{A(i)\sigma_i[C(i)]}|=(n_i-(p_i+r_i))!$, and
	\item $|G(i)_{A(i)B(i)}\cap G(i)_{A(i)\sigma_i[C(i)]}|=|G(i)_{A(i)B(i)\sigma_i[C(i)]}|=(n_i-(p_i+q_i+r_i))!$.
\end{itemize}
Then, \begin{align*}
|G(i)_{A(i)B(i)}G(i)_{A(i)\sigma_i[C(i)]}|&=|G(i)_{A(i)B(i)}| |G(i)_{A(i)\sigma_i[C(i)]}/G(i)_{A(i)B(i)\sigma_i[C(i)]}|\\
&=\frac{|G(i)_{A(i)B(i)}||G(i)_{A(i)\sigma_i[C(i)]}|}{|G(i)_{A(i)B(i)\sigma_i[C(i)]}|}\\
&=\frac{(n_i-(p_i+q_i))!(n_i-(p_i+r_i))!}{(n_i-(p_i+q_i+r_i))!}.
\end{align*}
Therefore, we have that \begin{align*}
D_i&=\frac{|G(i)_{A(i)B(i)}G(i)_{A(i)\sigma_i[C(i)]}|}{|G(i)_{A(i)}|}\\
&= \frac{\prod_{0\le k<r_i}\limits(n_i-(p_i+q_i+k))}{\prod_{0\le k<r_i}\limits (n_i-(p_i+k))}=d_i.
\end{align*}
By the choice of $(n_i)$, $\prod D_i$ is non-zero, and $\sigma C\indo_A^{\mu} B$ where $\sigma:=(\sigma_i)\in G$.
\end{proof}
From the proof of Claim \ref{extension_axiom_(X,G)}, $\mu$-independence of $(X,G)$ is described as follows: For $C=\{(c_{ij})|\ j< |C|\}$, $A=\{(a_{ij}|\ j<|A|)\}$, and $B=\{(b_{ij}|\ j<|B|)\}$, $C\depo^{\mu}_A B$ if and only if the set $\{i|\ |C(i)\setminus A(i)B(i)|<|C(i)\setminus A(i)|\}$ is infinite. Let $p=|A|$, $q=|B|$, and $r=|C|$. For each $i$, let $p_i:=|A(i)|$, $q_i=|B(i)\setminus A(i)|$, $r_i=|C(i)\setminus A(i)|$, and $r_i'=|C(i)\setminus A(i)B(i)|$. WLOG we may assume that for all $i$, $n_i>p+q+r$. Then we have that
\begin{align*}
D_i'&:=\frac{|G(i)_{A(i)B(i)}G(i)_{A(i)C(i)}|}{|G(i)_{A(i)}|}\\
&=\frac{(n_i-(p_i+q_i))!(n_i-(p_i+r_i))!}{(n_i-p_i)!(n_i-(p_i+q_i+r_i'))!}\\
&=\frac{\prod_{0\le k< r_i'}\limits (n_i-(p_i+q_i+k))}{\prod_{0\le k<r_i}\limits (n_i-(p_i+k))}.
\end{align*}
Therefore, we have that $C\depo^{\mu}_A B$ if and only if $\prod D_i'$ is zero if and only if $\{i|\ r_i'<r_i\}$ is infinite.

For any countable infinite set $I$, there is a sequence $(I_k)$ of infinite subsets of $I$ such that $I_k$ is disjoint from $\bigcup_{j< k}I_j$ for each $k$. Fix a such sequence $(I_k)$. Then for any finite set $C\subset X$, there is a sequence $(A_k)$ of subset of $X$ such that $I_k=\{i|\ |C(i)\setminus \bigcup_{j\le k} A_j(i)|<|C(i)\setminus \bigcup_{j< k} A_j(i)|\}$ so that $C\depo^{\mu}_{A_{<j}} A_k$ for $A_{<k}=\bigcup_{j< k}A_j$. We conclude that $\mR(X,G)=\infty$.
\end{example}

\begin{question}\label{unsolved_question}
	For a small compact structure, $\indo^{\mu}=\indo^{nm}$?
\end{question}

\noindent In \cite[Proposition 3.7]{N1}, Tanovic noticed that $nm$-independence and $\mu$-independence coincide in small $nm$-stable profinite structures. By adapting his proof, we show the same holds for small $nm$-stable compact structures. 
\begin{remark}\label{rem:cpt_str_equiv_mu_nm_independence}
Let $(X,G)$ be a small $nm$-stable compact structure. Then $\indo^{\mu}=\indo^{nm}$.
\end{remark}
\begin{proof}
Suppose $\indo^{\mu}\neq \indo^{nm}$. Take tuples $a,b\in X$ and a finite subset $A\subset X$ such that $a\indo^{\mu}_A b$ and $a\depo^{nm}_A b$ with $NM(a/b)$ small as possible. For the simplicity, we assume that $A=\emptyset$. 
\begin{claim}\label{claim1}
Let $b'\in \o(b)$ and let $a'\in \o(a)$ such that $a'b'\in \o(ab)$. Suppose $b\indo^{nm} b'$.  Then for any $a''\in \o(a/b)\cap \o(a'/b')$, $a''\depo^{nm}_{b'}bb'$.
\end{claim}
\begin{proof}
Take $a''\in \o(a/b)\cap \o(a'/b')$ arbitrary. If $a''\indo^{nm}_{b'}bb'$, then by transitivity, $a''\indo^{nm}b$. Since $a''\in \o(a/b)$, $a\indo^{nm}b$, which is  a contradiction.
\end{proof}
\noindent By Theorem \ref{description_mu_independence}, there is a finite quasi-invariant Borel measure $\nu$ on $(\o(a),G)$ such that $a\indo^{\mu}B\Leftrightarrow \o(a/B)\subset_{\nu}\o(a).$
\begin{claim}\label{claim2}
$\nu(\o(a/b)\cap\o(a'/b'))=0$.
\end{claim}
\begin{proof}
Suppose $\nu(\o(a/b)\cap\o(a'/b'))>0$. By $\sigma$-additivity, there is $a''\in \o(a/b)\cap\o(a'/b')$ such that $\nu(\o(a''/bb'))>0$, and $a''\indo^{\mu}bb'$. By Claim \ref{claim1}, we have $a''\depo^{nm}_{b'} bb'$ and $a''\depo^{nm} bb'$. So we have that $$NM(a''/bb')<NM(a''/b)=NM(a/b),$$ which contradicts with the minimality of $NM(a/b)$.
\end{proof}
\noindent Take $n\ge 1$ such that $n\alpha>\nu(\o(a))$. Take $b_0,b_1,\ldots, b_{n-1}\in\o(b)$ such that $b_i\indo^{nm}b_0\ldots b_{i-1}$ for $i\le n-1$, and choose $a_0,\ldots,a_{n-1}\in \o(a)$ such that $a_ib_i\in\o(ab)$. By Claim \ref{claim1} and \ref{claim2}, we have $$\nu(\o(a))\ge \nu(\bigcup_i\limits  \o(a_i/b_i))=\sum_{i}\nu(\o(a_i/b_i))=n\alpha >\nu (\o(a)),$$ which is impossible.

\end{proof}

\begin{question}\label{question:compare_mu_nm}
\begin{enumerate}
	\item For a small $nm$-stable $\mu$-structure, $\indo^{\mu}=\indo^{nm}$?
	\item For a small $\mu$-stable $\mu$-structure, $\indo^{\mu}=\indo^{nm}$?
	\item Let $(H,G)$ be a small, $\mu$-stable compact $\mu$-group such that $G$ is compact. Is $G$ a Polish compact group?
\end{enumerate}

\end{question}

\appendix

\section{Profinite branch groups give small actions\\ \textnormal{Michael Cohen and Phillip Wesolek}}

\subsection{Preliminaries}
A \textbf{rooted tree} $T$ is a locally finite tree with a distinguished vertex $r$ called the \textbf{root}. Letting $d$ be the usual graph metric, the \textbf{levels} of $T$ are the sets $V_n:=\{v\in T\mid d(v,r)=n\}$. The \textbf{children} of a vertex $v\in V_n$ is collection of $w\in V_{n+1}$ such that there is an edge from $v$ to $w$.

When vertices $k$ and $w$ lie on the same path to the root and $d(k,r)\leq d(w,r)$, we write $k\leq w$. Given a vertex $s\in T$, the \textbf{tree below $s$}, denoted $T^s$, is the collection of $t$ such that $s\leq t$ equipped with the induced graph structure.

We call a rooted tree \textbf{spherically homogeneous} if all $v$ and $w$ in $V_n$ the number of children of $v$ is the same as the number of children of $w$. A spherically homogeneous tree is completely determined by specifying the number of children of the vertices at each level. These data are given by an infinite sequence $\alpha\in \BN^{\BN}$ such that $\alpha(i)\geq 2$ for all $i\in \BN$;  the condition $\alpha(i)\geq 2$ ensures that there are no ``only children." We denote a spherically homogeneous tree by $T_{\alpha}$ for $\alpha\in \BN_{\geq 2}^{\BN}$. 

Profinite branch groups are certain closed subgroups of $\Aut(T_{\alpha})$; our approach to branch groups follows closely Grigorchuk's presentation in \cite{G}. For $G\leq \Aut(T_{\alpha})$ a closed subgroup and for a vertex $v\in T_{\alpha}$, the \textbf{rigid stabilizer of $v$} in $G$ is defined to be
\[
\rist_G(v):=\{g\in G\mid g(w)=w \text{ for all }w\in T_{\alpha}\setminus T_{\alpha}^v\}.
\]
The rigid stabilizer acts non-trivially only on the subtree $T^v_{\alpha}$. 

The \textbf{$n$-th rigid level stabilizer} in $G$ is defined to be
\[
\rist_G(n):=\grp{\rist_G(v)\mid v\in V_n}.
\]
It is easy to see that $\rist_G(n)\simeq \prod_{v\in V_n}\rist_G(v)$, and as a consequence, $\rist_G(n)$ is a closed subgroup of $G$.

\begin{definition}\label{df:profinitebranch}
A profinite group $G$ is said to be a \textbf{profinite branch group} if there is a tree $T_{\alpha}$ for some $\alpha\in \BN_{\geq 2}^{\BN}$ such that the following hold:
\begin{enumerate}
\item $G$ is isomorphic to a closed subgroup of $\Aut(T_{\alpha})$. 
\item $G$ acts transitively on each level of $T_{\alpha}$.
\item For each level $n$, the index $|G:\rist_G(n)|$ is finite.
\end{enumerate}
\end{definition}
\noindent We always identify a profinite branch group $G$ with the isomorphic closed subgroup of $\Aut(T_{\alpha})$. 

The rigid level stabilizers form a basis at $1$ for the topology on a profinite branch group $G$. The transitivity of the action on the levels ensures that $\rist_G(v)\simeq \rist_G(w)$ for all $v$ an $w$ in $V_n$. The transitivity further insures that profinite branch groups are always infinite.

Rooted trees admit a boundary; we restrict our definitions to the trees $T_{\alpha}$, although this is unnecessary. The \textbf{boundary} of $T_{\alpha}$, denoted by $\partial T_{\alpha}$, is the collection of infinite sequences $(v_1,v_2,\dots)$ such that $v_1=r$ and $v_{i+1}$ is a child of $v_i$. The boundary admits a canonical topology. For a finite sequence $\overline{s}:=(s_1,\dots,s_n)$ such that $s_1=r$ and $s_{i+1}$ is a child of $s_i$ for $1\leq i<n$, define 
\[
\Sigma_{\overline{s}}:=\{(v_i)_{i\in \BN}\in \partial T_{\alpha}\mid v_i=s_i\text{ for } 1\leq i\leq n\}.
\]
 The collection of the sets $\Sigma_{\overline{s}}$ where $\overline{s}$ ranges of all finite sequences of the described form gives a basis for a topology on $\partial T_{\alpha}$. This topology makes $\partial T_{\alpha}$ into a Cantor space.
 
Given a profinite branch group $G\leq \Aut(T_{\alpha})$, the action of $G$ on $T_{\alpha}$ clearly extends to a continuous action by homeomorphisms on $\partial T_{\alpha}$. We call the action of $G$ on $\partial T_{\alpha}$ the \textbf{boundary action}.

\subsection{The boundary action is small}

\begin{lemma}\label{lem:rist_orbits}
For $G\leq \Aut(T_{\alpha})$ a profinite branch group and $v\in T_{\alpha}$, $\rist_G(v)$ has finitely many orbits on $\partial T_{\alpha}^v$.
\end{lemma}
\begin{proof}
Since $G$ acts spherically transitively on $T_{\alpha}$, the point stabilizer $G_{v}$ acts spherically transitively on $T_{\alpha}^v$. Let $K$ be the kernel of the action of $G_{v}$ on $T^v_{\alpha}$. Supposing that $v$ is on level $n$, the rigid level stabilizer $\rist_G(n)$ is a finite index subgroup of $G_{v}$, as $G$ is a branch group. Furthermore, $\rist_G(w)\leq K$ for all $w\neq v$, so $K\rist_G(v)$ is of finite index in $G_{v}$. Say that $K\rist_G(v)$ is of index $s$ in $G_{v}$. 

Since $G_{v}$ acts  transitively on each level of $T_{\alpha}^v$ and $\rist_G(v)$ is normal in $G(v)$, it follows that $\rist_G(v)$ has at most $s$ many orbits on each level of $T_{\alpha}^v$. Let $m\geq n$ be the level such that $\rist_G(v)$ has the largest number of orbits on the $m$-th level of $T_{\alpha}^v$. Say $O_1,\dots,O_l$ lists the orbits of $\rist_G(v)$. For any $w$ and $w'$ on a level $k>m$ such that $w$ and $w'$ are descendants of some $o$ and $o'$ in $O_j$, we deduce that $w$ and $w'$ lie in the same orbit of $\rist_G(v)$ on level $k$, since else we contradict the maximality of $m$. We conclude that $\rist_G(v)$ has exactly $l$ orbits on every level $k\geq m$.

For each $O_i$, let $\partial O_i=\bigcup_{w\in O_i}\partial T^w_{\alpha}$. We argue that $\rist_G(v)$ acts transitively  on each $\partial O_i$. Take $\delta=(x_1,x_2,\dots)$ and $\xi=(y_1,y_2,\dots) $ in $O_i$. Let $k\geq m$ be the first level such that $x_k\neq y_k$. Since $x_k$ and $y_k$ are descendants of elements of $O_i$, there is some $g_1\in \rist_G(v)$ such that $g_1(x_k)=y_k$.  Setting $\delta_1:=g_1(\delta)$, we see that $\delta_1$ and $\xi$ share the first $k$ coordinates. Continuing in this fashion, we produce a sequence $g_i\in \rist_G(v)$ such that $g_i(\delta)\rightarrow \xi$. 

The subgroup $\rist_G(v)$ is compact, so by passing to a subsequence, we may assume that $g_i\rightarrow g$ with $g\in \rist_G(v)$. The element $g$ must be such that $g(\delta)=\xi$. Hence, $\rist_G(v)$ acts transitively on $\partial O_i$.
\end{proof}

\begin{theorem} 
For $G\leq \Aut(T_{\alpha})$ a profinite branch group, the action of $G$ on $\partial T_{\alpha}$ is small, the $\mu$-structure $(\partial T_{\alpha},G)$ has $\mu$-rank $1$, and $(\partial T_{\alpha},\mcl)$ forms a trivial pregeometry.
\end{theorem}
\begin{proof}
Let $F\subseteq \partial T_{\alpha}$ be finite. Say that $F=\{\xi_1,\dots,\xi_n\}$ with $\xi_j=(x_1^j,x_2^j,\dots)$ and set 
\[
\Omega:=\{v\in T_{\alpha}\mid \exists i,j\; d(v,x_i^j)=1\text{ and }\forall i,j\; v\neq x_i^j \}.
\]

For any $\delta \in \partial T_{\alpha}\setminus F$, there is some $v\in \Omega$ and $i\geq 1$ such that the $i$-th coordinate of $\delta $ equals $v$. It follows that 
\[
\bigcup_{v\in \Omega}\partial T^v_{\alpha}=\partial T_{\alpha} \setminus F.
\]
Furthermore, $\rist_G(v)\leq G_{F}$ for any $v\in \Omega$, and in view of Lemma~\ref{lem:rist_orbits}, $\rist_G(v)$ has finitely many orbits on $\partial T^v_{\alpha}$. The subgroup $\grp{\rist_G(v)\mid v\in \Omega}$ of $G_{F}$ therefore has countably many orbits on $\partial T_{\alpha} \setminus F$. We deduce that $G_{F}$ has countably many orbits on $\partial T_{\alpha}$. Hence, the action of $G$ on $\partial T_{\alpha}$ is small.

Let us now show that the $\mu$-structure $(\partial T_{\alpha},G)$ has $\mu$-rank $1$ and forms a pregeometry. We argue that for $\delta \in \partial T_{\alpha}$ and $F\subset \partial_{\alpha} T$ finite, $G_{\delta}G_{F}$ contains an open set if $\delta \notin F$. Let $F=\{\xi_1,\ldots,\xi_n\}$. Write $\delta=(y_1,y_2,\ldots)$ and $\xi_i=(x_1^i,x_2^i,\ldots)$ for $i=1,\ldots,n$. Since $\xi_1,\ldots,\xi_n,$ and $\delta$ are distinct, there is a level $m$ such that $y_m$ and $x_m^1,\ldots,x_m^n$ are distinct in $V_m$. Then $\rist_G(y_m)\subset G_{F}$ and $\rist_G(v)\subset G_{\delta}$ for all $v\neq y_m\in V_m$. Hence, $\rist_G(m)\subset G_{\delta}G_F$, so $G_{\delta}G_F$ contains an open set. On the other hand, $G_F$ does not contain $\rist_G(n)$ for any $n$. Indeed, fix $n$ and observe that $\rist_G(n)$ is normal in $G$. The group $G$ acts transitively on each level, so if $\rist_G(n)$ fixes a vertex on level $k$, then it in fact fixes the entire level. We conclude that $\rist_G(n)$ acts without fixed points on all suitably deep levels. It now follows that $\rist_G(n)$ is not contained in $G_F$ for any $n$. Hence, $\mu_G(G_F)=0$, since positive measure subgroups are open. For $\delta \in \partial T_{\alpha}$ and $F\subset \partial T_{\alpha}$ finite, we conclude that
\begin{itemize}
	\item $\delta \in \mcl(F)$ if and only if $\delta\in F$; and
	\item $\delta\indo^{\mu} F$ if and only if $\delta\notin F$.
\end{itemize} 
Thus, $\mR(\partial T_{\alpha})=1$ and $(\partial T_{\alpha},\mcl^{eq})$ forms a pregeometry.
\end{proof}


\begin{thebibliography}{99}
\bibitem{Bour}
N. Bourbaki.
\newblock Integration II: Chapters 7-9,
\newblock Springer-Verlag Berlin Heidelberg, (2004).

\bibitem{DE}
A. Deitmar and S. Echterhoff.
\newblock Principles of Harmonic Analysis,
\newblock Springer International Publishing, (2014).


\bibitem{G}
R. Grigorchuk.
\newblock Just infinite branch groups, in New horizons in pro-{$p$} groups,
\newblock Springer Science+Buisness Media, (2000), 121-179.

\bibitem{KT}
E. Kaniuth and K. F. Taylor.
\newblock Induced representations of locally compact groups,
\newblock Cambridge University Press, (2012)

\bibitem{Kim}
B. Kim.
\newblock Simplicity Theory,
\newblock Oxford University Press, (2014).

\bibitem{Ke}
A. Kechris.
\newblock Classical Descriptive Set Theory,
\newblock Springer, (1995).

%


\bibitem{Kru4}
K. Krupinksi.
\newblock Generalizations of small profinite structures,
\newblock {\it Journal of Symbolic Logic}, {\bf 75} (2010), 1147-1175.

\bibitem{Kru5}
K. Krupinski.
\newblock Some model theory of Polish structures,
\newblock {\it Tans. Amer. Math. Soc.}, {\bf 362} (2010), 3499-3533.

\bibitem{KW}
K. Krupinski and F. Wagner.
\newblock Small, nm-stable compact G-groups,
\newblock {\it Israel Journal of Mathematics}, {\bf 194} (2013), 907-933. 

\bibitem{N1}
L. Newelski.
\newblock $\mathfrak{M}$-gap conjecture and m-normal theories,
\newblock {\it Israel Journal of Mathematics}, {\bf 106} (1998), 285-311.
%

\bibitem{N2}
L. Newelski.
\newblock Small profinite structures,
\newblock {\it Tans. Amer. Math. Soc.}, {\bf 354} (2001), 925-943.

\bibitem{Wagner}
F. O. Wagner.
\newblock Simple Theories,
\newblock Springer Netherlands, (2000).
\end{thebibliography}
\end{document}